\documentclass[12pt]{article}
\usepackage{amsfonts}
\usepackage{graphicx}
\usepackage[pdfdisplaydoctitle=true]{hyperref}

\usepackage{amssymb}
\usepackage{epstopdf}
\usepackage{framed}
\usepackage{fullpage}
\usepackage{color}
\usepackage{setspace}
\usepackage{amsmath,amsthm}

\newcommand{\Uq}{\mathcal{U}_q(\mathfrak{g})}

\newtheorem{theorem}{Theorem}

\newtheorem{lemma}{Lemma}
\newtheorem{corollary}{Corollary}
\theoremstyle{remark}

\author{Jeffrey Kuan} 
\title{An explicit central element of $\mathcal{U}_q(\mathfrak{so}_5)$ and its corresponding quantum Hamiltonian} 
\date{}

\begin{document}

\maketitle

\begin{framed}

Accessibility Statement: An accessible version of this PDF, which meets the guidelines of WCAG Level 2.1AA, can be found on the first author's webpage at this link:\\

\url{https://www.math.tamu.edu/~jkuan/CentralElementB2_Accessible.pdf}

\end{framed}

\abstract{A previous paper of the author developed a general method for producing explicit central elements of quantized Lie algebras using Lusztig's inner product. This method had previously been applied for the type $C_2,D_3$ and $D_4$ Lie algebras. The current paper repeats the calculation for the type $B_2$ Lie algebra, which is actually isomorphic to the $C_2$ Lie algebra. The explicit expression for the corresponding quantum Hamiltonian is computed.}

\section{Introduction}
In mathematical physics, quantum Hamiltonians can be constructed from central elements of quantization of Lie algebras. In order to compute the entries of the Hamiltonian, these central elements need to be explicitly written in terms of the generators. To this end, the author proved a formula in \cite{Kua16} to produce explicit expressions for central elements, using Lusztig's inner product. The method had found explicit central elements for the Lie algebras with Dynkin diagrams $C_2$ and $D_3,D_4$ \cite{Kua16,KLLPZ20} (for the type $A$ Lie algebras, the central element had already been found in \cite{GZB91} and used in \cite{Kua17}, but their method does not appear to apply here). In the present paper, we carry about the calculation for $B_2$. Note that the $B_2$ and $C_2$ Lie algebras are isomorphic, but whereas the $C_2$ calculation uses the four--dimensional representation ($\mathfrak{sp}_4$), this paper uses the five--dimensional representation ($\mathfrak{so}_5$). Indeed, by the Harish--Chandra isomorphism, the center of the $B_2 \cong C_2$ Lie algebra is generated by two elements. These two elements can be taken to be the central elements of \cite{Kua16} and the one here. See also \cite{KuaZha18} for an analog of the explicit central elements generating the center in the case of the $A_n$ Lie algebras.

In the calculation done in this paper, we observe that not every element of the dual basis needs to be calculated, and that the automorphisms of the Lie algebra streamline some of the computations. This will slightly reduce computational work in the future.

From the central element, the co--product of the quantized Lie algebra produces a quantum Hamiltonian. If all the matrix entries of the quantum Hamiltonian are non--negative, then a ground state conjugation will define a Markov process, which can be interpreted as an interacting particle system. We will also provide the matrix entries of the quantum Hamiltonian.

If the entries of the quantum Hamiltonian are non--negative, then a ground state transformation will produce a Markov process. This ideas goes back to \cite{SchSan94} and \cite{Sch97}. More recently it was generalized in \cite{CGRS15} (see also \cite{CGRS16}). If some of the entries are negative, it is still sometimes possible to remove such states, as in \cite{Kua16} or \cite{KLLPZ20}.

\section{Background}

\subsection{Description of type $B_2 \cong C_2$ Lie algebra}
Define the symplectic Lie algebra $\mathfrak{sp}_4$ to be the Lie algebra consisting of $4\times 4$ matrices of the form (where $A,B,C,D$ are $2\times 2$ blocks)
$$
\left\{
\left(
\begin{tabular}{cc}
$A $& $B $\\
$C$ & $D$
\end{tabular}
\right)
: A=-D^T, B=B^T, C=C^T
\right\}
$$
Using the notation that $E_{ij}$ is a matrix with a $1$ in the $(i,j)$--entry and $0$ everywhere else, set 
\begin{align*}
e_1&= E_{12} - E_{43}, \quad  f_1 = E_{21} - E_{34}, \quad h_1 = E_{11} - E_{22} -  E_{33} + E_{44}  \\
e_2 &= E_{24}, \quad  f_2 = E_{42}, \quad  h_2 = E_{22} - E_{44} 
\end{align*}
Then $\mathfrak{sp}_4$ has a basis $e_1, e_2, e_1e_2, e_2e_1, e_1e_2e_1, f_1, f_2, f_1f_2, f_2f_1, f_1f_2f_1, h_1, h_2$. Each $e_i,f_i,h_i$ generate a copy of $\mathfrak{sl}_2,$ and the two copies are related by
$$
[h_1,e_2]=-2e_2, \quad [h_1,f_2]=2f_2, \quad [2h_2,e_1]=-2e_1, \quad [2h_2,f_1]=2f_1,
$$
where we have written $2h_2$ to preserve the symmetry. As usual, these relations can be summarized by describing the roots of $\mathfrak{sp}_4$. Let $\mathfrak{h}$ denote the linear span of $h_1,h_2$ and identify $\mathfrak{h} \cong \mathbb{R}^2$ via $h_1\mapsto x_1,2h_2\mapsto x_2$. Define $\alpha_1(x_1)=2,\alpha_1(x_2)=-2$ and $\alpha_2(x_1)=-2,\alpha_2(x_2)=4$ so that the above relations are
$$
[x_i,e_j]=\alpha_j(x_i)e_j \quad [x_i,f_j]=-\alpha_j(x_i)f_j
$$
If $\mathfrak{h}^*$ is identified with $\mathbb{R}^2$ then $\alpha_1=(1,-1),\alpha_2=(0,2)$, so that $\alpha_j(x_i)=(\alpha_i,\alpha_j)$ where $(\cdot,\cdot)$ is the usual Euclidean inner product.

Let $V$ be the natural four--dimensional representation of $\mathfrak{sp}_4$. Using the identification of $\mathfrak{h}^*$ with $\mathbb{R}^2$, show that $V$ has a basis $v_1,v_2,v_4,v_3$ which are in the weight spaces $V[(1,0)],V[(0,1)], V[(0,-1)], V[(-1,0)]$ respectively.

\subsection{A five--dimensional representation}
Let $V$ be the four--dimensional fundamental representation of $\mathfrak{sp}_4$. The exterior power $V \wedge V$ is a representation of dimension $\binom{4}{2}=6$. The  action of $\mathfrak{sp}_4$ preserves the subspace $\mathbb{C}v$ where $v=v_1 \wedge v_3 - v_2 \wedge v_4$, and therefore preserves the five--dimensional quotient space $W:=(V\wedge V)/\mathbb{C}v$. 

Define the basis $\{w_1,\ldots,w_5\}$ by
\begin{align*}
w_1 &= v_1 \wedge v_2 \\
w_2 &= v_1 \wedge v_4 \\
w_3 &= v_1 \wedge v_3 + v_2 \wedge v_4 \\
w_4 &= v_2 \wedge v_3\\ 
w_5 &= {\color{black} v_4 \wedge v_3}
\end{align*}
Then (again using the identification of $\mathfrak{h}^*$ with $\mathbb{R}^2$)
\begin{align*}
W[(1,1)] &= \mathbb{C} w_1  \\
W[(1,-1)] &= \mathbb{C} w_2 \\
W[(0,0)] &= \mathbb{C} w_3  \\
W[(-1,1)] &= \mathbb{C}w_4 \\ 
W[(-1,-1)] &= \mathbb{C}w_5 
\end{align*}
The representation of $\mathfrak{sp}_4$ onto $W$ actually defines an isomorphism $\mathfrak{sp}_4 \cong \mathfrak{so}_5$, which can also be seen in the Dynkin diagram.

\subsection{Quantum group}
The quantization of a finite--dimensional simple Lie algebra $\mathfrak{g}$ depends on its Dynkin diagram. Rather than define $\mathcal{U}_q(\mathfrak{g})$ most generally, here is the definition of $\mathcal{U}_q(\mathfrak{sp}_4)$. It is generated by $\{e_i,f_i,k_i\},i=1,2$ with the Weyl relations
$$
[e_i,f_j] = \delta_{ij}\frac{k_i-k_i^{-1}}{q_i-q_i^{-1}}, \quad [k_i,k_j]=0
$$
$$
k_{i}e_{j}= q^{(\alpha_i,\alpha_j)}e_{j}k_{i} \quad k_{i}f_{j}= q^{-(\alpha_i,\alpha_j)}f_{j}k_{i} , \quad 1\leq i,j\leq 2
$$
(where $q_1=q,q_2=q^2$) and the Serre relations {\color{black}
\begin{align*}
e_2^2e_1 - ( q^2 + q^{-2}  )e_2e_1e_2 + e_1e_2^2  &= 0  \\
e_1^3e_2 - (q^2 + 1 + q^{-2})e_1^2e_2e_1 + (q^2+1+q^{-2})e_1e_2e_1^2 - e_2e_1^3 &= 0  \\
f_2^2f_1 - ( q^2 + q^{-2}  )f_2f_1f_2 + f_1f_2^2  &= 0  \\
f_1^3f_2 - (q^2 + 1 + q^{-2})f_1^2f_2f_1 + (q^2+1+q^{-2})f_1f_2f_1^2 - f_2f_1^3 &= 0  \\
\end{align*}}
The co--product is
$$
\Delta(e_i) = e_i \otimes  1  + k_i \otimes  e_i \quad \Delta(f_i) = 1\otimes  f_i + f_i\otimes  k_i^{-1}, \quad \Delta(k_i) = k_i \otimes  k_i,
$$

\subsection{Central Element Construction}\label{Central Element}
Here, we explain how to construct a central element of $\mathcal{U}_q(\mathfrak{g})$.

Letting $\mathfrak{b}_{\pm}\subset \mathfrak{g}$ denote the Borel subalgebras (that is, $\mathfrak{b}_+$ is generated by $e_i,h_i$ and $\mathfrak{b}_-$ is generated by $f_i,h_i$), there is a bi--linear pairing (see Proposition 6.12 of \cite{J}) on $ \mathcal{U}_q(\mathfrak{b}_-) \times \mathcal{U}_q(\mathfrak{b}_+)$ defined on generators by 
$$
\langle k_{\alpha}, k_{\beta}\rangle = q^{-(\alpha,\beta)_{\mathfrak{g}}}, \quad \langle f_i,e_j\rangle = \frac{-\delta_{ij}}{q_i-q_i^{-1}}, \quad \langle k_i,e_j\rangle = \langle f_i,k_j\rangle = \langle 1,e_i, \rangle=\langle f_j,1\rangle=0, \quad \langle 1, 1\rangle=1
$$
and extended to all of  $ \mathcal{U}_q(\mathfrak{b}_-) \times \mathcal{U}_q(\mathfrak{b}_+)$ by 
\begin{equation}\label{Extended}
\langle y, xx' \rangle = \langle \Delta(y), x'\otimes  x\rangle, \quad \langle yy', x\rangle = \langle y \otimes  y', \Delta(x)\rangle, \quad \langle y\otimes  y' , x\otimes  x'\rangle = \langle y,x\rangle\langle y',x'\rangle.
\end{equation}
where $( \cdot,\cdot )_{\mathfrak{g}}$ is the invariant, non--degenerate invariant symmetric bilinear form on $\mathfrak{h}^*$. Furthermore, according to Lemma 6.16 of \cite{J},
\begin{equation}\label{anti}
\langle \omega(x),\omega(y) \rangle = \langle y,x \rangle = \langle \tau(y),\tau(x)\rangle 
\end{equation}
where $\omega$ is the automorphism and $\tau$ is the antiautomorphism defined by
\begin{align*}
\omega(e_i)=f_i, \quad \omega(f_i)=e_i, \quad \omega(k_i)=k_i^{-1}\\
\tau(e_i)=e_i, \quad \tau(f_i)=f_i, \quad \tau(k_i)=k_i^{-1}
\end{align*}

{\color{black}  
Choose an ordering $\geq$ on the weight space. Given $\nu \geq 0$, let $r(\nu)$ be the dimension of $U[\nu]$. Let $\{u_{\nu}^i\}_{ 1 \leq i \leq r(\nu)}$ be an arbitrary basis of $U[\nu]$, and let $\{v_{\nu}^i\}_{ 1 \leq i \leq r(\nu)}$ be the dual basis of $U[-\nu]$ under $\langle \cdot , \cdot \rangle$.
}

Let $V$ be a fundamentalrepresentation of $\mathfrak{g}$ and let $\rho$ be half the sum of the positive roots of $\mathfrak{g}$, where a root $\alpha$ is positive if $\alpha>0$.  Explicitly, $\rho=(2,1)$ for $\mathfrak{sp}_4$. {\color{black} Let $\{v_{\lambda}\}$ be a basis of $V$ where each $v_{\lambda} \in V[\lambda]$, and let $f_{\lambda}$ be a dual basis. }

The next lemma is Lemma 2.1 of \cite{Kua16}, which constructs a central element of $\mathcal{U}_q(\mathfrak{g})$. Recall that the root lattice of $\mathfrak{g}$ is the lattice in $\mathfrak{h}^*$ spanned by the roots of $\mathfrak{g}$. When $\mathfrak{g}=\mathfrak{sp}_4$, the root lattice can be explicitly written as  $\{(x_1,x_2): x_1+x_2 \in 2\mathbb{Z}\}\in \mathbb{Z}^2$.
{\color{black} \begin{lemma}\label{CentralLemma}
If $q$ is not a root of unity and $2\mu$ is in the root lattice of $\mathfrak{g}$ for all weights $\mu$ of $V$, then the element
\begin{equation}\label{central}
\sum_{\mu \geq\lambda } \sum_{i,j=1}^{r(\mu-\lambda)}  q^{(\mu-\lambda,\mu)}  q^{-(2\rho,\mu)} f_{\lambda}(v^j_{\mu-\lambda}u^i_{\mu-\lambda} v_{\lambda} ) v^j_{\mu-\lambda} k_{-\lambda-\mu} u^i_{\mu-\lambda}
\end{equation}
is central in $\Uq$, where the sum is taken over all $\mu,\lambda$ such that $V[\mu]$ and $V[\lambda]$ are nonzero. 
\end{lemma}
}

\section{Main Result and Proof}
\begin{theorem} The element $C$ defined by

\begin{multline*}
q^{-2}\left(q-q^{-1}\right)^2 \left(  (1-q^2)f_1f_2f_1f_2 + (q^4-q^{-2})f_2f_1^2f_2 + (1-q^2)f_2f_1f_2f_1 + (1-q^2)f_1f_2^2f_1  \right)  k_{(0,0)}     \\ 
\times \left(    (1-q^2)e_1e_2e_1e_2 + (q^4-q^{-2})e_2e_1^2e_2 + (1-q^2)e_2e_1e_2e_1 + (1-q^2)e_1e_2^2e_1     \right) \\
   +{\left(q-q^{-1}\right)^4 
   f_1f_1k_{(0,0)}e_1e_1}\\
   +\left(q-q^{-1}\right)^2 \left((1+q^2)f_1f_2
   f_1-f_2f_1f_1-q^2f_1f_1f_2 \right)k_{(0,2)}\left( (1+q^2)e_1e_2e_1 -e_1e_1e_2-q^2e_2e_1
  e_1 \right)\\
   +\left(q-q^{-1}\right)^2 \left(q+q^{-1}\right) \left(f_1f_2
   q^2-f_2f_1\right)k_{(1,1)}\left(e_2e_1 q^2-e_1e_2\right)\\
   +\left(q^2-q^{-2}\right)^2 q^4
   f_2k_{(2,0)}e_2\\
   +q^{-4}\left(q-q^{-1}\right)^2 \left( (1+q^2)f_1f_2f_1 - f_1f_1f_2 - q^2f_2f_1f_1\right)k_{(0,-2)}\left((1+q^2)e_1e_2e_1 - e_2e_1e_1 -q^2e_1e_1e_2\right)\\
   +q^{-4}\left(q-q^{-1}\right)^2 \left(q+q^{-1}\right) \left(q^2 f_2f_1
   -f_1f_2\right)k_{(-1,-1)}\left(q^2 e_1e_2 -e_2e_1\right)+q^{-4}\left(q^2-q^{-2}\right)^2
   f_2k_{(-2,0)}e_2\\
   +\left(q-q^{-1}\right)^2 \left(q+q^{-1}\right) f_1k_{(-1,1)}e_1+\left(q-q^{-1}\right)^2
   \left(q+q^{-1}\right) f_1k_{(1,-1)}e_1\\
   +q^6 k_{(2,2)}+q^{-6}k_{(-2,-2)}+q^2 k_{(2,-2)}+q^{-2}k_{(-2,2)}+k_{(0,0)}
\end{multline*}
is central.
\end{theorem}

\begin{corollary}
The action of 
$$
\left(\frac{1}{q^{5}}-\frac{1}{q^{3}}-q^{3}+q^{5}\right)^{-1}\left(C-\left(1+\frac{1}{q^{10}}+\frac{1}{q^{6}}+q^{6}+q^{10}\right)\mathrm{Id}\right)
$$
on the tensor power of two $4$--dimensional representations is given by the following $16\times 16$ matrix: 

\includegraphics[height=3in]{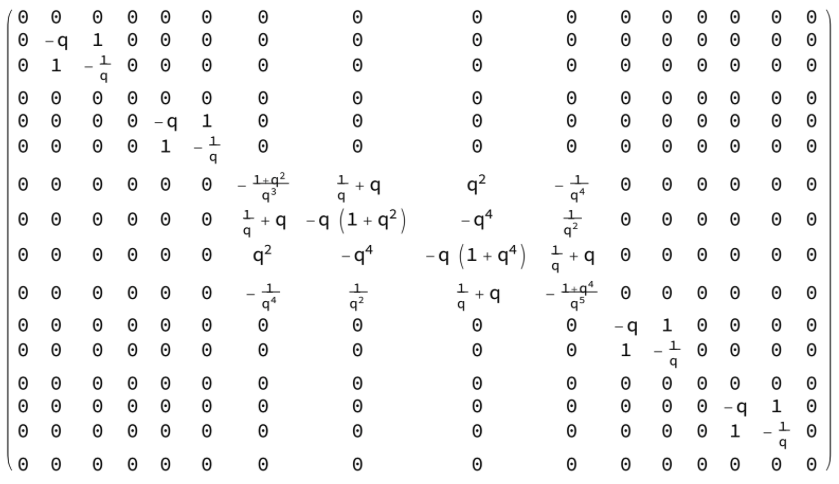}
    
\end{corollary}

\subsection{Proof}

We wish to apply Lemma 1. The first step is to find the dual elements. Most of them were found in Lemma 2.2 of \cite{Kua16}. There is an additional dual element that needs to be found, due to our use of the five--dimensional representation rather than the four--dimensional representation. The remaining one is stated in the following lemma.
\begin{lemma}
1. The set $\{e_2^2e_1,e_1e_2^2\}$ is a basis of $U[(1,3)]$, and the dual of $e_2^2e_1$ is 
$$
(q^{-2}-q^2)(q^2+q^{-2})^{-1}(e_2e_1)^*f_2  .
$$

2. The (ordered) set $\mathcal{B} = \{e_1e_2e_1e_2, e_2e_1^2e_2, e_1e_2^2e_1, e_2e_1e_2e_1 \}$ is a basis of $U[(2,2)]$. Its dual basis is the (ordered) set
\begin{multline*}
\mathcal{B}^* = \{ - (e_1e_1e_2)^*f_2 + f_2(e_1e_1e_2)^*, q^{-2}(e_1e_1e_2)^*f_2 - q^2 f_2 (e_1e_1e_2)^*, \\
 (q+q^{-1})^{-1}((e_2e_2e_1)^*f_1 - q^2 f_1 (e_2e_2e_1)^* ), - f_2(e_2e_1e_1)^* + (e_2e_1e_1)^*f_2 \},
\end{multline*}
where
$$
f_2(e_2e_1e_1)^* = \frac{q-q^{-1}}{q+q^{-1}}  \left((q^2-q^4)f_2f_1f_2f_1 - f_2f_1^2f_2 + q^2f_1f_2^2f_1 \right) 
$$
$$
(e_2e_1e_1)^* f_2=  \frac{q-q^{-1}}{q+q^{-1}} \left( (1-q^{-2})f_1f_2f_1f_2 + f_1f_2^2f_1   - q^2f_2f_1^2f_2\right)
$$
and applying $\tau$
$$
f_2(e_1e_1e_2)^* = \frac{q-q^{-1}}{q+q^{-1}}  \left((1-q^{-2})f_2f_1f_2f_1+ f_1f_2^2f_1 -q^2f_2f_1^2f_2 \right) 
$$
$$
(e_1e_1e_2)^* f_2 = \frac{q-q^{-1}}{q+q^{-1}}  \left((q^2-q^4)f_1f_2f_1f_2 - f_2f_1^2f_2 + q^2f_1f_2^2f_1 \right) 
$$
\end{lemma}
\begin{proof}
1. By Serre's relations, $U[(1,3)]$ is two--dimensional (rather than three--dimensional).  It is immediate that
$$
\langle  (e_2e_1)^*f_2 ,  e_1 e_2^2\rangle = \langle  f_2(e_1e_2)^* ,   e_2^2e_1\rangle  = 0.
$$
because $e_2e_1$ cannot appear in the left tensor power of $\Delta(e_1e_2e_2)$. One the needs only compute that
\begin{align*}
\langle  (e_2e_1)^*f_2 ,  e_2^2e_1 \rangle  &= \langle (e_2e_1)^* \otimes  f_2, \Delta(e_2^2e_1) \rangle\\
&=  \langle (e_2e_1)^* \otimes  f_2, e_2k_2e_1 \otimes  e_2 + k_2e_2e_1 \otimes  e_2 \rangle\\
&=  (1+q^4)\langle (e_2e_1)^* \otimes  f_2, e_2k_2e_1 \otimes  e_2  \rangle\\
&=  (1+q^4)q^{-2}\langle (e_2e_1)^* \otimes  f_2, e_2e_1k_2 \otimes  e_2  \rangle\\
&= (q^2+q^{-2}) (q^{-2}-q^2)^{-1} .
\end{align*}

2. By Serre's relations, $U[(2,2)]$ is four--dimensional (rather than six--dimensional).

Finding the first two elements amounts to showing that
\begin{align*}
\langle - (e_1e_1e_2)^*f_2 + f_2(e_1e_1e_2)^*, e_1e_2e_1e_2 \rangle &= 1 \\
\langle - (e_1e_1e_2)^*f_2 + f_2(e_1e_1e_2)^*, e_2e_1^2e_2 \rangle &= 0 \\
\langle - (e_1e_1e_2)^*f_2 + f_2(e_1e_1e_2)^*, e_1e_2^2e_1 \rangle &= 0 \\
\langle - (e_1e_1e_2)^*f_2 + f_2(e_1e_1e_2)^*, e_2e_1e_2e_1 \rangle &= 0 \\
\langle q^{-2}(e_1e_1e_2)^*f_2 - q^2 f_2 (e_1e_1e_2)^* , e_1e_2e_1e_2 \rangle &= 0 \\
\langle q^{-2}(e_1e_1e_2)^*f_2 - q^2 f_2 (e_1e_1e_2)^* , e_2e_1^2e_2 \rangle &= 1 \\
\langle q^{-2}(e_1e_1e_2)^*f_2 - q^2 f_2 (e_1e_1e_2)^* , e_1e_2^2e_1 \rangle &= 0 \\
\langle q^{-2}(e_1e_1e_2)^*f_2 - q^2 f_2 (e_1e_1e_2)^* , e_2e_1e_2e_1 \rangle &= 0 
\end{align*}
Observe that the third, fourth, seventh, and eighth lines are immediate, since
$$
\langle (e_1e_1e_2)^*, e_1e_2e_1 \rangle = \langle (e_1e_1e_2)^*, e_2e_1e_1 \rangle = 0.
$$
For example,
$$
\langle  (e_1e_1e_2)^*f_2, e_2e_1e_2e_1 \rangle = \langle (e_1e_1e_2)^* \otimes  f_2, \Delta\left( e_2e_1e_2e_1 \right) \rangle 
$$
equals zero because the term $e_1e_1e_2$ cannot appear in $ \Delta\left( e_2e_1e_2e_1 \right)$.

This leaves four pairings to compute, which are
\begin{align*}
\langle (e_1e_1e_2)^*f_2,  e_1e_2e_1e_2 \rangle  &= \langle (e_1e_1e_2)^* \otimes f_2,  \Delta(e_1e_2e_1e_2) \rangle \\
&= \langle (e_1e_1e_2)^* \otimes f_2,  e_1k_2e_1e_2 \otimes  e_2 \rangle \\
&= q^2 \langle (e_1e_1e_2)^* , (e_1e_1e_2) k_2\rangle (q^{-2}-q^2)^{-1} \\
&= q^2 (q^{-2}-q^2)^{-1} 
\end{align*}
and
\begin{align*}
\langle f_2(e_1e_1e_2)^*,  e_1e_2e_1e_2 \rangle  &= \langle f_2 \otimes (e_1e_1e_2)^* ,  \Delta(e_1e_2e_1e_2) \rangle \\
&= \langle f_2 \otimes (e_1e_1e_2)^* ,  k_1e_2k_1k_2 \otimes  e_1e_1e_2 \rangle \\
&= \langle f_2, k_1e_2k_1k_2  \rangle \\
&= q^{-2} \langle f_2, e_2k_1k_1k_2  \rangle\\
&=q^{-2} (q^{-2}-q^2)^{-1} 
\end{align*}
and
\begin{align*}
\langle (e_1e_1e_2)^*f_2,  e_2e_1^2e_2 \rangle  &= \langle (e_1e_1e_2)^* \otimes f_2,  \Delta(e_2e_1^2e_2 ) \rangle \\
&= \langle (e_1e_1e_2)^* \otimes f_2,  k_2e_1^2e_2 \otimes  e_2 \rangle \\
&=  \langle (e_1e_1e_2)^* , (e_1e_1e_2) k_2\rangle (q^{-2}-q^2)^{-1} \\
&=  (q^{-2}-q^2)^{-1} 
\end{align*}
and
\begin{align*}
\langle f_2(e_1e_1e_2)^*,  e_2e_1^2e_2  \rangle  &= \langle f_2 \otimes (e_1e_1e_2)^* ,  \Delta(e_2e_1^2e_2 ) \rangle \\
&= \langle f_2 \otimes (e_1e_1e_2)^* ,  e_2k_1k_1k_2 \otimes  e_1e_1e_2 \rangle \\
&= \langle f_2, e_2k_1k_1k_2  \rangle \\
&=(q^{-2}-q^2)^{-1} .
\end{align*}

Finding the fourth element is easy by applying $\tau$ to the first four elements above. One gets
\begin{align*}
\langle - f_2(e_2e_1e_1)^* + (e_2e_1e_1)^*f_2, e_2e_1e_2e_1 \rangle &= 1 \\
\langle -  f_2(e_2e_1e_1)^* + (e_2e_1e_1)^*f_2, e_2e_1^2e_2 \rangle &= 0 \\
\langle -  f_2(e_2e_1e_1)^* +(e_2e_1e_1)^*f_2, e_1e_2^2e_1 \rangle &= 0 \\
\langle -  f_2(e_2e_1e_1)^* + f(e_2e_1e_1)^*f_2, e_1e_2e_1e_2 \rangle &= 0 \\
\end{align*}

Now, finding the third element amounts to showing that
\begin{align*}
\langle (e_2e_2e_1)^*f_1 - q^2 f_1 (e_2e_2e_1)^* , e_2e_1e_2e_1 \rangle &= 0 \\
\langle (e_2e_2e_1)^*f_1 - q^2 f_1 (e_2e_2e_1)^* , e_1e_2^2e_1 \rangle &= q+q^{-1} \\
\langle (e_2e_2e_1)^*f_1 - q^2 f_1 (e_2e_2e_1)^* , e_2e_1^2e_2 \rangle &= 0 \\
\langle (e_2e_2e_1)^*f_1 - q^2 f_1 (e_2e_2e_1)^* , e_1e_2e_1e_2 \rangle &= 0 
\end{align*}
The third and fourth  lines follow from similar reasoning as above. This leaves two pairings to compute, which are
\begin{align*}
\langle (e_2e_2e_1)^*f_1,  e_1e_2^2e_1 \rangle  &= \langle (e_2e_2e_1)^* \otimes f_1,  \Delta(e_1e_2^2e_1 ) \rangle \\
&= \langle (e_2e_2e_1)^* \otimes f_1,  k_1e_2^2e_1 \otimes  e_1 \rangle \\
&=  q^{-2}\langle (e_2e_2e_1)^* , (e_2e_2e_1) k_1\rangle (q^{-1}-q)^{-1} \\
&=  q^{-2}(q^{-1}-q)^{-1} 
\end{align*}
and
\begin{align*}
\langle f_1(e_2e_2e_1)^*,  e_1e_2^2e_1  \rangle  &= \langle f_1 \otimes (e_2e_2e_1)^* ,  \Delta(e_1e_2^2e_1 ) \rangle \\
&= \langle f_1 \otimes (e_2e_2e_1)^* ,  e_1k_2k_2k_1 \otimes  e_2e_2e_1 \rangle \\
&= \langle f_1, e_1k_2k_2k_1  \rangle \\
&=(q^{-1}-q)^{-1} .
\end{align*}

This finishes the proof of the lemma.
\end{proof}

The next step is to write the representation matrices. 

\begin{lemma}

\begin{multline*}
f_2 = \left(
\begin{array}{ccccc}
 0 & 0 & 0 & 0 & 0 \\
 1 & 0 & 0 & 0 & 0 \\
 0 & 0 & 0 & 0 & 0 \\
 0 & 0 & 0 & 0 & 0 \\
 0 & 0 & 0 & 1 & 0 \\
\end{array}
\right), 
\quad 
f_1 = \left(
\begin{array}{ccccc}
 0 & 0 & 0 & 0 & 0 \\
 0 & 0 & 0 & 0 & 0 \\
 0 & 1 & 0 & 0 & 0 \\
 0 & 0 & q+q^{-1} & 0 & 0 \\
 0 & 0 & 0 & 0 & 0 \\
\end{array}
\right),
\quad
e_2 = \left(
\begin{array}{ccccc}
 0 & 1 & 0 & 0 & 0 \\
 0 & 0 & 0 & 0 & 0 \\
 0 & 0 & 0 & 0 & 0 \\
 0 & 0 & 0 & 0 & 1 \\
 0 & 0 & 0 & 0 & 0 \\
\end{array}
\right), \\
e_1 = \left(
\begin{array}{ccccc}
 0 & 0 & 0 & 0 & 0 \\
 0 & 0 & q+q^{-1} & 0 & 0 \\
 0 & 0 & 0 & 1 & 0 \\
 0 & 0 & 0 & 0 & 0 \\
 0 & 0 & 0 & 0 & 0 \\
\end{array}
\right), 
\quad
k_2 = \left(
\begin{array}{ccccc}
 q^2 & 0 & 0 & 0 & 0 \\
 0 & q^{-2} & 0 & 0 & 0 \\
 0 & 0 & 1 & 0 & 0 \\
 0 & 0 & 0 & q^2 & 0 \\
 0 & 0 & 0 & 0 & q^{-2} \\
\end{array}
\right), 
\quad 
k_1 = 
\left(
\begin{array}{ccccc}
 1 & 0 & 0 & 0 & 0 \\
 0 & q^2 & 0 & 0 & 0 \\
 0 & 0 & 1 & 0 & 0 \\
 0 & 0 & 0 & q^{-2} & 0 \\
 0 & 0 & 0 & 0 & 1 \\
\end{array}
\right)
\end{multline*}
\end{lemma}
\begin{proof}
By using the five--dimensional representation and extending to the quantum group, the result follows. One can also check the relations directly.
\end{proof}

By applying the previous two lemmas and applying Lemma 1, this computes every term in the central element except for the first one. Because the terms $v_3,v_4$ were not explicitly found, the term is 
\begin{multline*}
q^{-2}\left(q-q^{-1}\right)^2 \left(  (1-q^2)f_1f_2f_1f_2 + (q^4-q^{-2})f_2f_1^2f_2 + (1-q^2)f_2f_1f_2f_1 + (1-q^2)f_1f_2^2f_1  \right)  k_{(0,0)}     \\ 
\times \left(    (1-q^2)e_1e_2e_1e_2 + (q^4-q^{-2})e_2e_1^2e_2 + Ae_2e_1e_2e_1 + B e_1e_2^2e_1     \right) 
\end{multline*}
for some $A,B$. Since the bases were arbitrary, by using \eqref{anti} with the automorphism $\omega$, it can be seen that $A=B=(1-q^2)$. Another way to find $A$ and $B$ is to use the following explicit representation. 

One can check that the action on $\mathbb{C}^4 \otimes  \mathbb{C}^4$ is given by the $16\times 16$ matrices 
\begin{align*}
e_1 &= E_{12} + E_{34} + E_{67} + E_{10,12} - E_{59} - E_{8,11} - E_{13,14} - E_{15,16} \\
&+ q E_{13} + q^{-1} E_{24} + q E_{58} - q E_{6,10} - q^{-1} E_{7,12} + q^{-1} E_{9,11} - q E_{13,15} - q^{-1}E_{14,16} \\
e_2 &= E_{25} + E_{48} + E_{7,13} + E_{12,15} + E_{3,6} + q^2 E_{47} + q^{-2} E_{8,13} + E_{11,14} \\
f_1 &= E_{31} + E_{42} + E_{85} - E_{10,6} - E_{12,7} + E_{11,9} - E_{15,13} - E_{16,14} \\
&+ q^{-1} E_{21} + q E_{43} + q^{-1} E_{76} + q E_{12,10} - q^{-1} E_{95} - q E_{11,8} - q^{-1} E_{14,13} - q E_{16,15}\\
f_2 &= E_{63} + E_{74} + E_{13,8} + E_{14,11} + E_{52} + q^{-2} E_{84} + q^2 E_{13,7} + E_{15,12}\\
k_{(a,b)} &= q^{2a} E_{11} + q^{a+b}(E_{22} + E_{33}) + q^{2b} E_{44} + q^{a-b}(E_{55}+E_{66}) + E_{77} + E_{88} + E_{99} + E_{10,10} \\
&+ q^{b-a} (E_{11,11} + E_{12,12}) + q^{-2b} E_{13,13} + q^{-b-a}(E_{14,14} + E_{15,15}) + q^{-2a} E_{16,16} 
\end{align*}
and that the action of the element only commutes with the other generators when $A=B=1-q^2$. This action only provides a way to double--check that the answer is correct. 

Note that one could also check that the action on the four and five--dimensional fundamental representations is a multiple of the identity, but this would not determine the values of $A$ and $B$.

Finally, the corollary follows by explicitly calculating the action of $C$ on $\mathbb{C}^4 \otimes \mathbb{C}^4$, using the above representation.


\end{document}